\documentclass{amsart}
\usepackage{amssymb}
\usepackage{amsmath}
\usepackage{amsfonts}

\newtheorem{theorem}{Theorem}
\theoremstyle{plain}

\newtheorem{corollary}{Corollary}

\newtheorem{definition}{Definition}
\newtheorem{example}{Example}

\newtheorem{lemma}{Lemma}

\newtheorem{proposition}{Proposition}
\newtheorem{remark}{Remark}

\numberwithin{equation}{section}
\input{tcilatex}

\begin{document}
\title{ON SLANT RIEMANNIAN SUBMERSIONS FOR COSYMPLECTIC MANIFOLDS}
\author{I. K\"{u}peli Erken}
\address{Art and Science Faculty,Department of Mathematics, Uludag
University, 16059 Bursa, TURKEY}
\email{iremkupeli@uludag.edu.tr}
\author{C. Murathan}
\address{Art and Science Faculty,Department of Mathematics, Uludag
University, 16059 Bursa, TURKEY}
\email{cengiz@uludag.edu.tr}
\date{14.11.2013}
\subjclass[2000]{Primary  53C43, 53C55; Secondary 53D15}
\keywords{Riemannian submersion, cosymplectic manifold, slant submersion\\
\ \ \ This paper is supported by Uludag University research project
(KUAP(F)-2012/57).}

\begin{abstract}
In this paper we introduce slant Riemannian submersions from cosymplectic
manifolds onto Riemannian manifolds. We obtain some results on slant
Riemannian submersions of a cosymplectic manifolds. We also give examples
and inequalities between the scalar curvature and squared mean curvature of
fibres of such slant submersions according to characteristic vector field is
vertical or horizontal.
\end{abstract}

\maketitle

\section{\textbf{Introduction}}

Riemannian submersions were introduced in the sixties by B. O'Neill and A.
Gray (see \cite{BO1}, \cite{GRAY}) as a tool to study the geometry of a
Riemannian manifold with an additional structure in terms of certain
components, that is, the fibers and the base space. The Riemannian
submersions are of great interest both in mathematics and physics, owing to
their applications in the Yang-Mills theory (\cite{BL}, \cite{WATSON}),
Kaluza-Klein theory (\cite{BOL}, \cite{IV}), supergravity and superstring
theories (\cite{IV2}, \cite{MUS}), etc. Riemannian submersions were
considered between almost complex manifolds by Watson in \cite{WAT} under
the name of almost Hermitian submersion. For Riemannian submersions between
almost contact manifolds, Chinea \cite{CHI} studied under the name of almost
contact submersions. \c{S}ahin \cite{SAHIN} introduced a kind of submersion
which was defined from almost Hermitian manifolds to any Rieamannian
manifolds. Recently there are several kinds of submersions according to the
conditions on it: e.g., contact-complex submersion \cite{IANUS},
quaternionic submersion \cite{IANUS2}, almost $h$-slant submersion and $h$%
-slant submersion \cite{PARK1}, semi-invariant submersion \cite{SAHIN2}, $h$%
-semi-invariant submersion \cite{PARK2}, etc.

On the other hand, The study of slant submanifolds was initiated by B. Y.
Chen in \cite{CHEN}. In \cite{SAHIN1}, B. \c{S}ahin studied slant
submersions from almost Hermitian manifold to a Riemannian manifold and
generalized his results which was given in \cite{SAHIN}.

In this paper, we study slant Riemannian submersions from cosymplectic
manifolds.We obtain some results on slant Riemannian submersions of a
cosymplectic manifolds.

The paper is organized in the following way. In section 2, we recall some
notions needed for this paper. Section 3 deals with cosymplectic manifolds.
In section 4, we give definition of slant Riemannian submersions and
introduce slant Riemannian submersions from cosymplectic manifolds onto
Riemannian manifolds. We survey main results of slant submersions defined on
cosymplectic manifolds and obtain some interesting properties about them. We
construct examples of slant submersions such that characteristic vector
field $\xi $ is vertical or horizontal. We give a sufficient condition for a
slant Riemannian submersion from cosymplectic manifolds onto Riemannian
manifolds to be harmonic. Moreover, we investigate the geometry of leaves of
\ $(\ker F_{\ast })$ and $(\ker F_{\ast })^{\bot }$. Here, we find a
necessary and sufficient condition for a slant Riemannian submersion to be
totally geodesic. We give sharp inequalities between the scalar curvature
and squared mean curvature of fibres such that characteristic vector field $%
\xi $ is vertical or horizontal. Moreover, we know that anti-invariant
submersions are special slant submersions with slant angle $\theta =\frac{%
\pi }{4}.$ We investigated such a submersions in \cite{KM}. Especially, we
will give some additional results for anti-invariant submersions from a
cosymplectic manifold to a Riemannian manifold such that $(\ker F_{\ast
})^{\bot }$ =$\phi ($ $\ker (F_{\ast }))\oplus \{\xi \}$.

\section{\textbf{Riemannian Submersions}}

In this section we recall several notions and results which will be needed
throughout the paper.

Let $(M,g_{M})$ be an $m$-dimensional Riemannian manifold , let $(N,g_{N})$
be an $n$-dimensional Riemannian manifold. A Riemannian submersion is a
smooth map $F:M\rightarrow N$ which is onto and satisfies the following
axioms:

$S1$. $F$ has maximal rank.

$S2$. The differential $F_{\ast }$ preserves the lenghts of horizontal
vectors.

The fundamental tensors of a submersion were defined by O'Neill (\cite{BO1},%
\cite{BO2}). They are $(1,2)$-tensors on $M$, given by the formula:%
\begin{eqnarray}
\mathcal{T}(E,F) &=&\mathcal{T}_{E}F=\mathcal{H}\nabla _{\mathcal{V}E}%
\mathcal{V}F+\mathcal{V}\nabla _{\mathcal{V}E}\mathcal{H}F,  \label{AT1} \\
\mathcal{A}(E,F) &=&\mathcal{A}_{E}F=\mathcal{V}\nabla _{\mathcal{H}E}%
\mathcal{H}F+\mathcal{H}\nabla _{\mathcal{H}E}\mathcal{V}F,  \label{AT2}
\end{eqnarray}%
for any vector field $E$ and $F$ on $M.$ Here $\nabla $ denotes the
Levi-Civita connection of $(M,g_{M})$. These tensors are called
integrability tensors for the Riemannian submersions. Note that we denote
the projection morphism on the distributions ker$F_{\ast }$ and (ker$F_{\ast
})^{\perp }$ by $\mathcal{V}$ and $\mathcal{H},$ respectively. The following
\ Lemmas are well known (\cite{BO1},\cite{BO2}).

\begin{lemma}
\label{T}For any $U,W$ vertical and $X,Y$ horizontal vector fields, the
tensor fields $\mathcal{T}$, $\mathcal{A}$ satisfy:%
\begin{eqnarray}
i)\mathcal{T}_{U}W &=&\mathcal{T}_{W}U,  \label{TUW} \\
ii)\mathcal{A}_{X}Y &=&-\mathcal{A}_{Y}X=\frac{1}{2}\mathcal{V}\left[ X,Y%
\right] .  \label{TUW2}
\end{eqnarray}
\end{lemma}

It is easy to see that $\mathcal{T}$ $\ $is vertical, $\mathcal{T}_{E}=%
\mathcal{T}_{\mathcal{V}E}$ and $\mathcal{A}$ is horizontal, $\mathcal{A=A}_{%
\mathcal{H}E}$.

For each $q\in N,$ $F^{-1}(q)$ is an $(m-n)$ dimensional submanifold of $M$.
The submanifolds $F^{-1}(q),$ $q\in N,$ are called fibers. A vector field on 
$M$ is called vertical if it is always tangent to fibers. A vector field on $%
M$ is called horizontal if it is always orthogonal to fibers. A vector field 
$X$ on $M$ is called basic if $X$ is horizontal and $F$-related to a vector
field $X$ on $N,$ i. e., $F_{\ast }X_{p}=X_{\ast F(p)}$ for all $p\in M.$

\begin{lemma}
\label{g}Let $F:(M,g_{M})\rightarrow (N,g_{N})$ be a Riemannian submersion.
If $\ X,$ $Y$ are basic vector fields on $M$, then:
\end{lemma}

$i)$ $g_{M}(X,Y)=g_{N}(X_{\ast },Y_{\ast })\circ F,$

$ii)$ $\mathcal{H}[X,Y]$ is basic, $F$-related to $[X_{\ast },Y_{\ast }]$,

$iii)$ $\mathcal{H}(\nabla _{X}Y)$ is basic vector field corresponding to $%
\nabla _{X_{\ast }}^{^{\ast }}Y_{\ast }$ where $\nabla ^{\ast }$ is the
connection on $N.$

$iv)$ for any vertical vector field $V$, $[X,V]$ is vertical.

Moreover, if $X$ is basic and $U$ is vertical then $\mathcal{H}(\nabla
_{U}X)=\mathcal{H}(\nabla _{X}U)=\mathcal{A}_{X}U.$ On the other hand, from (%
\ref{AT1}) and (\ref{AT2}) we have%
\begin{eqnarray}
\nabla _{V}W &=&\mathcal{T}_{V}W+\hat{\nabla}_{V}W  \label{1} \\
\nabla _{V}X &=&\mathcal{H\nabla }_{V}X+\mathcal{T}_{V}X  \label{2} \\
\nabla _{X}V &=&\mathcal{A}_{X}V+\mathcal{V}\nabla _{X}V  \label{3} \\
\nabla _{X}Y &=&\mathcal{H\nabla }_{X}Y+\mathcal{A}_{X}Y  \label{4}
\end{eqnarray}%
for $X,Y\in \Gamma ((\ker F_{\ast })^{\bot })$ and $V,W\in \Gamma (\ker
F_{\ast }),$ where $\hat{\nabla}_{V}W=\mathcal{V}\nabla _{V}W.$ On any fibre 
$F^{-1}(q),$ $q\in N,\hat{\nabla}$ coincides with the Levi-Civita connection
with respect to the metric induced by $g_{M}.$ This induced metric on fibre $%
F^{-1}(q)$ is denoted by $\hat{g}.$

Notice that $\mathcal{T}$ \ acts on the fibres as the second fundamental
form of the submersion and restricted to vertical vector fields and it can
be easily seen that $\mathcal{T}=0$ is equivalent to the condition that the
fibres are totally geodesic. A Riemannian submersion is called a Riemannian
submersion with totally geodesic fiber if $\mathcal{T}$ vanishes
identically. Let $U_{1},...,U_{m-n}$ be an orthonormal frame of $\Gamma
(\ker F_{\ast }).$ Then the horizontal vector field $H$ $=\frac{1}{m-n}%
\dsum\limits_{j=1}^{m-n}\mathcal{T}_{U_{j}}U_{j}$ is called the mean
curvature vector field of the fiber. If \ $H$ $=0$ the Riemannian submersion
is said to be minimal. A Riemannian submersion is called a Riemannian
submersion with totally umbilical fibers if 
\begin{equation}
\mathcal{T}_{U}W=g_{M}(U,W)H  \label{4a}
\end{equation}%
for $U,W\in $ $\Gamma (\ker F_{\ast })$. For any $E\in \Gamma (TM),\mathcal{T%
}_{E\text{ }}$and $\mathcal{A}_{E}$ are skew-symmetric operators on $(\Gamma
(TM),g_{M})$ reversing the horizontal and the vertical distributions. By
Lemma \ref{T} horizontally distribution $\mathcal{H}$ is integrable if and
only if \ $\mathcal{A=}0$. For any $D,E,G\in \Gamma (TM)$ one has%
\begin{equation}
g(\mathcal{T}_{D}E,G)+g(\mathcal{T}_{D}G,E)=0,  \label{4b}
\end{equation}%
\begin{equation}
g(\mathcal{A}_{D}E,G)+g(\mathcal{A}_{D}G,E)=0.  \label{4c}
\end{equation}%
The tensor fields $\mathcal{A}$, $\mathcal{T}$ and their covariant
derivatives play a fundamental role in expressing the Riemannian curvature $%
R $ of $(M,g)$. By (\ref{1}) and (\ref{2}) B.O'Neill (\cite{BO1}) gave%
\begin{equation}
R(S,W,V,U)=g(R(S,W)V,U)=\hat{R}(S,W,V,U)+g(\mathcal{T}_{U}W,\mathcal{T}%
_{V}S)-g(\mathcal{T}_{V}W,\mathcal{T}_{U}S),  \label{4D}
\end{equation}%
where $\hat{R}$ is Riemannian curvature tensor of any fibre $(F^{-1}(q),\hat{%
g}_{q})$. Precisely, if $\{U,V\}$ is orthonormal basis of the vertical $2$%
-plane then the equation (\ref{4D}) allow following%
\begin{equation}
K(U\wedge V)=\hat{K}(U\wedge V)+\parallel \mathcal{T}_{U}V\parallel ^{2}-g(%
\mathcal{T}_{U}U,\mathcal{T}_{V}V),  \label{4E}
\end{equation}%
where $K$ and $\hat{K}$ denote the sectional curvature of $M$ and fibre $%
F^{-1}(q)$ respectively. Moreover, the following formula stated in (\cite%
{BO1})%
\begin{eqnarray}
R(Y,W,V,X) &=&g((\nabla _{X}\mathcal{T})(V,W),Y)+g((\nabla _{V}\mathcal{A}%
)(X,Y),W)  \label{4F} \\
&&-g(\mathcal{T}_{V}X,\mathcal{T}_{W}Y)+g(\mathcal{A}_{X}V,\mathcal{A}_{V}Y),
\notag
\end{eqnarray}%
for any $X,Y,Z\in \Gamma ((\ker F_{\ast })^{\bot })$, $V,W\in \Gamma (\ker
F_{\ast })$.

We recall the notion of harmonic maps between Riemannian manifolds. Let $%
(M,g_{M})$ and $(N,g_{N})$ be Riemannian manifolds and suppose that $\varphi
:M\rightarrow N$ is a smooth map between them. Then the differential $%
\varphi _{\ast }$ of $\varphi $ can be viewed a section of the bundle $\
Hom(TM,\varphi ^{-1}TN)\rightarrow M,$ where $\varphi ^{-1}TN$ is the
pullback bundle which has fibres $(\varphi ^{-1}TN)_{p}=T_{\varphi (p)}N,$ $%
p\in M.\ Hom(TM,\varphi ^{-1}TN)$ has a connection $\nabla $ induced from
the Levi-Civita connection $\nabla ^{M}$ and the pullback connection. Then
the second fundamental form of $\varphi $ is given by 
\begin{equation}
(\nabla \varphi _{\ast })(X,Y)=\nabla _{X}^{\varphi }\varphi _{\ast
}(Y)-\varphi _{\ast }(\nabla _{X}^{M}Y)  \label{5}
\end{equation}%
for $X,Y\in \Gamma (TM),$ where $\nabla ^{\varphi }$ is the pullback
connection. It is known that the second fundamental form is symmetric. If $%
\varphi $ is a Riemannian submersion it can be easily prove that 
\begin{equation}
(\nabla \varphi _{\ast })(X,Y)=0  \label{5a}
\end{equation}%
for $X,Y\in \Gamma ((\ker F_{\ast })^{\bot })$.A smooth map $\varphi
:(M,g_{M})\rightarrow (N,g_{N})$ is said to be harmonic if $trace(\nabla
\varphi _{\ast })=0.$ On the other hand, the tension field of $\varphi $ is
the section $\tau (\varphi )$ of $\Gamma (\varphi ^{-1}TN)$ defined by%
\begin{equation}
\tau (\varphi )=div\varphi _{\ast }=\sum_{i=1}^{m}(\nabla \varphi _{\ast
})(e_{i},e_{i}),  \label{6}
\end{equation}%
where $\left\{ e_{1},...,e_{m}\right\} $ is the orthonormal frame on $M$.
Then it follows that $\varphi $ is harmonic if and only if $\tau (\varphi
)=0 $, for details, \cite{B}.

\section{\textbf{Cosymplectic Manifolds}}

A $(2m+1)$-dimensional $C^{\infty }$-manifold $M$ is said to have an almost
contact structure if there exist on $M$ a tensor field $\phi $ of type $%
(1,1) $, a vector field $\xi $ and 1-form $\eta $ satisfying:%
\begin{equation}
\phi ^{2}=-I+\eta \otimes \xi ,\text{ \ }\phi \xi =0,\text{ }\eta \circ \phi
=0,\text{ \ \ }\eta (\xi )=1.  \label{fi}
\end{equation}%
There always exists a Riemannian metric $g$ on an almost contact manifold $M$
satisfying the following conditions%
\begin{equation}
g(\phi X,\phi Y)=g(X,Y)-\eta (X)\eta (Y),\text{\ }\eta (X)=g(X,\xi )
\label{metric}
\end{equation}%
where $X,Y$ are vector fields on $M.$

An almost contact structure $(\phi ,\xi ,\eta )$ is said to be normal if the
almost complex structure $J$ on the product manifold $M\times R$ is given by%
\begin{equation*}
J(X,f\frac{d}{dt})=(\phi X-f\xi ,\eta (X)\frac{d}{dt})
\end{equation*}%
where $f$ is the $C^{\infty }$-function on $M\times R$ has no torsion i.e., $%
J$ is integrable. The condition for normality in terms of $\phi ,\xi $ and $%
\eta $ is $\left[ \phi ,\phi \right] +2d\eta \otimes \xi =0$ on $M$, where $%
\left[ \phi ,\phi \right] $ is the Nijenhuis tensor of $\phi .$~Finally, the
fundamental two-form $\Phi $ is defined by $\Phi (X,Y)=g(X,\phi Y).$

An almost contact metric structure $(\phi ,\xi ,\eta ,g)$ is said to be
cosymplectic, if it is normal and both $\Phi $ and $\eta $ are closed (\cite%
{BL1}, \cite{LUDDEN}), and the structure equation of a cosymplectic manifold
is given by%
\begin{equation}
(\nabla _{E}\phi )G=0  \label{nambla}
\end{equation}%
for any $E,G$ tangent to $M,$ where $\nabla $ denotes the Riemannian
connection of the metric $g$ on $M.$ Moreover, for cosymplectic manifold%
\begin{equation}
\nabla _{E}\xi =0.  \label{xzeta}
\end{equation}%
The $\phi $-\textit{sectional curvature} of a cosymplectic manifold $M$
defined for any unit vector $E$ on $M$ orthogonal to $\xi $ by the formula 
\begin{equation}
H(E)=g(R(E,\phi E)\phi E,E).  \label{HOLOMORPHIC}
\end{equation}%
If cosymplectic manifold $M$ \ has $\phi $-\textit{sectional curvature} $c$,
then $M$ is called a \textit{cosymplectic space} \textit{form} and denoted
by $M(c)$ \cite{LUDDEN}. So the curvarure tensor $R$ of a cosymplectic space
form $M(c)$ is given by \cite{LUDDEN}%
\begin{eqnarray}
R(X,Y)Z &=&\frac{c}{4}(g(Y,Z)X-g(X,Z)Y+\eta (X)\eta (Z)Y-\eta (Y)\eta
(Z)X+g(X,Z)\eta (Y)\xi )  \notag \\
&&-g(Y,Z)\eta (X)\xi +g(\phi Y,Z)\phi X-g(\phi X,Z)\phi Y-2g(\phi X,Y)\phi
Z),  \label{CURVATURE}
\end{eqnarray}%
for any tangent vector fields $X,Y,Z$ to $M(c).$

The canonical example of cosymplectic manifold is given by the product $%
B^{2n}\times 
\mathbb{R}
$ Kahler manifold $B^{2n}(J,g)$ with $%
\mathbb{R}
$ real line. Now we will introduce a well known cosymplectic manifold
example on $%
\mathbb{R}
^{2n+1}.$

\begin{example}[\protect\cite{O}]
\label{O}We consider $%
\mathbb{R}
^{2n+1}$ with Cartesian coordinates $(x_{i},y_{i},z)$ $(i=1,...,n)$ and its
usual contact form 
\begin{equation*}
\eta =dz,
\end{equation*}%
The characteristic vector field $\xi $ is given by $\frac{\partial }{%
\partial z}$ and its Riemannian metric $g$ and tensor field $\phi $ are
given by%
\begin{equation*}
g=\sum\limits_{i=1}^{n}((dx_{i})^{2}+(dy_{i})^{2})+(dz)^{2},\text{ \ }\phi
=\left( 
\begin{array}{ccc}
0 & \delta _{ij} & 0 \\ 
-\delta _{ij} & 0 & 0 \\ 
0 & 0 & 0%
\end{array}%
\right) \text{, \ }i=1,...,n
\end{equation*}%
This gives a cosymplectic structure on $%
\mathbb{R}
^{2n+1}$. The vector fields $E_{i}=\frac{\partial }{\partial y_{i}},$ $%
E_{n+i}=\frac{\partial }{\partial x_{i}}$, $\xi $ form a $\phi $-basis for
the cosymplectic structure. On the other hand, it can be shown that $%
\mathbb{R}
^{2n+1}(\phi ,\xi ,\eta ,g)$ is a cosymplectic manifold.
\end{example}

\begin{example}[\protect\cite{KIM}]
\label{KIM}We denote Cartesian coordinates in $%
\mathbb{R}
^{5}$ by $(x_{1},x_{2},x_{3},x_{4},x_{5})$ and its Riemannian metric $g$ 
\begin{equation*}
g=\left( 
\begin{array}{ccccc}
1+\tau ^{2} & 0 & \tau ^{2} & 0 & -\tau \\ 
0 & 1 & 0 & 0 & 0 \\ 
\tau ^{2} & 0 & 1+\tau ^{2} & 0 & -\tau \\ 
0 & 0 & 0 & 1 & 0 \\ 
-\tau & 0 & -\tau & 0 & 1%
\end{array}%
\right) ,\text{ \ }
\end{equation*}%
where $\tau =\sin (x_{1}+x_{3}).$ We define \ an \ almost contact structure $%
(\phi ,\xi ,\eta )$ on $%
\mathbb{R}
^{5}$ by 
\begin{equation*}
\phi =\left( 
\begin{array}{ccccc}
0 & -1 & 0 & 0 & 0 \\ 
1 & 0 & 0 & 0 & 0 \\ 
0 & 0 & 0 & -1 & 0 \\ 
0 & 0 & 1 & 0 & 0 \\ 
0 & -\tau & 0 & -\tau & 0%
\end{array}%
\right) ,\eta =-\tau dx_{1}-\tau dx_{3}+dx_{5},\xi =\frac{\partial }{%
\partial x_{5}}.
\end{equation*}%
The fundamental 2-form $\Phi $ have the form 
\begin{equation*}
\Phi =dx_{1}\wedge dx_{2}+dx_{3}\wedge dx_{4}.
\end{equation*}%
This gives a cosymplectic structure on $%
\mathbb{R}
^{5}$. If we take vector fields $E_{1}=\frac{\partial }{\partial x_{1}}+\tau 
\frac{\partial }{\partial x_{5}},E_{2}=\frac{\partial }{\partial x_{3}},\phi
E_{1}=E_{3}=\frac{\partial }{\partial x_{2}},\phi E_{2}=E_{4}=\frac{\partial 
}{\partial x4}$ and $E_{5}=\frac{\partial }{\partial x_{5}}$ then these
vector fields form a frame field in $%
\mathbb{R}
^{5}$.
\end{example}

\section{\textbf{Slant Riemannian submersions}}

\begin{definition}
\label{DEF}Let $M(\phi ,\xi ,\eta ,g_{M})$ be a cosymplectic manifold and $%
(N,g_{N})$ be a Riemannian manifold. A Riemannian submersion $F:M(\phi ,\xi
,\eta ,g_{M})\rightarrow $ $(N,g_{N})$ is said to be slant if for any non
zero vector $X\in \Gamma (\ker F_{\ast })-\{\xi \}$, the angle $\theta (X)$
between $\phi X$ and the space $\ker F_{\ast }$ is a constant (which is
independent of the choice of $p\in M$ and of $X\in \Gamma (\ker F_{\ast
})-\{\xi \}$). The angle $\theta $ is called the slant angle of the slant
submersion. Invariant and anti-invariant submersions are slant submersions
with $\theta =0$ and $\theta =\pi /2$, respectively. A slant submersion
which is not invariant nor anti-invariant is called proper submersion.
\end{definition}

Now we will introduce examples.

\begin{example}
\label{E3}$%
\mathbb{R}
^{5}$ has got a cosymplectic structure as in Example \ref{O} . Let $F:%
\mathbb{R}
^{5}\rightarrow 
\mathbb{R}
^{2}$ be a map defined by $F(x_{1},x_{2},y_{1},y_{2},z)=(\frac{1}{\sqrt{2}}%
(x_{1}-x_{2}),y_{2})$. Then, by direct calculations 
\begin{equation*}
\ker F_{\ast }=span\{V_{1}=E_{1},V_{2}=\frac{1}{\sqrt{2}}%
(E_{3}+E_{4}),V_{3}=\xi =E_{5}\}
\end{equation*}%
and 
\begin{equation*}
(\ker F_{\ast })^{\bot }=span\{H_{1}=\frac{1}{\sqrt{2}}(E_{3}-E_{4}),\text{ }%
H_{2}=E_{2}\}.
\end{equation*}%
Then it is easy to see that $F$ is a Riemannian submersion. Moreover, $\phi
V_{1}=E_{3}$ and $\phi V_{2}=-\frac{1}{\sqrt{2}}(E_{1}+E_{2})$ imply that $%
\left\vert g(\phi V_{1},V_{2})\right\vert =\frac{1}{\sqrt{2}}$. So $F$ is a
slant submersion with slant angle $\theta =\frac{\pi }{4}$ and $\xi $ is
vertical.
\end{example}

\begin{example}
\label{E4}$%
\mathbb{R}
^{5}$ has got a cosymplectic structure as in Example \ref{KIM}. Let $F:%
\mathbb{R}
^{5}\rightarrow 
\mathbb{R}
^{2}$ be a map defined by $F(x_{1},x_{2},y_{1},y_{2},z)=(\frac{1}{\sqrt{2}}%
(x_{1}-y_{1}),\frac{1}{\sqrt{2}}(x_{2}-y_{2}))$. Then, by direct
calculations 
\begin{equation*}
\ker F_{\ast }=span\{V_{1}=\frac{1}{\sqrt{2}}(E_{3}+E_{4}),V_{2}=\frac{1}{%
\sqrt{2}}(E_{1}+E_{2}),V_{3}=\xi =E_{5}\}
\end{equation*}%
and 
\begin{equation*}
(\ker F_{\ast })^{\bot }=span\{H_{1}=\frac{1}{\sqrt{2}}(E_{3}-E_{4}),\text{ }%
H_{2}=\frac{1}{\sqrt{2}}(E_{1}-E_{2})\}.
\end{equation*}%
Then it is easy to see that $F$ is a Riemannian submersion. Moreover, $\phi
V_{1}=-\frac{1}{\sqrt{2}}(E_{1}+E_{2})$ and $\phi V_{2}=\frac{1}{\sqrt{2}}%
(E_{3}+E_{4})$imply that $\left\vert g(\phi V_{1},V_{2})\right\vert =1$. So $%
F$ is a slant submersion with slant angle $\theta =0.$
\end{example}

\begin{example}
\label{HOR}$%
\mathbb{R}
^{5}$ has got a cosymplectic structure as in Example \ref{O} . Let $F:%
\mathbb{R}
^{5}\rightarrow 
\mathbb{R}
^{3}$ be a map defined by $F(x_{1},x_{2},y_{1},y_{2},z)=(\frac{1}{\sqrt{2}}%
(x_{1}-x_{2}),y_{2},z)$. Then, by direct calculations 
\begin{equation*}
\ker F_{\ast }=span\{V_{1}=E_{1},V_{2}=\frac{1}{\sqrt{2}}(E_{3}+E_{4})\}
\end{equation*}%
and 
\begin{equation*}
(\ker F_{\ast })^{\bot }=span\{H_{1}=\frac{1}{\sqrt{2}}(E_{3}-E_{4}),\text{ }%
H_{2}=E_{2},H_{3}=\xi \}.
\end{equation*}%
Then it is easy to see that $F$ is a Riemannian submersion. Moreover, $\phi
V_{1}=E_{3}$ and $\phi V_{2}=-\frac{1}{\sqrt{2}}(E_{1}+E_{2})$ imply that $%
\left\vert g(\phi V_{1},V_{2})\right\vert =\frac{1}{\sqrt{2}}$. So $F$ is a
slant submersion with slant angle $\theta =\frac{\pi }{4}$and $\xi $ is
horizontal.
\end{example}

\subsection{Slant Riemannian submersions admitting vertical structure vector
field}

In this subsection, we will investigate the properties of slant Riemannian
submersions from cosymplectic manifolds onto a Riemannian manifold such that
characteristic vector field $\xi $ is vertical.

Now, let $F$ be a slant Riemannian submersion from a cosymplectic manifold $%
M(\phi ,\xi ,\eta ,g_{M})$ onto a Riemannian manifold $(N,g_{N}).$ Then for
any $U,V\in \Gamma (\ker F_{\ast }),$ we put 
\begin{equation}
\phi U=\psi U+\omega U,  \label{TAN}
\end{equation}%
where $\psi U$ and $\omega U$ are vertical and horizontal components of $%
\phi U$, respectively. Similarly, for any $X\in \Gamma (\ker F_{\ast
})^{\perp }$, we have 
\begin{equation}
\phi X=\mathcal{B}X+\mathcal{C}X,  \label{NOR}
\end{equation}%
where $\mathcal{B}X$ (resp.$\mathcal{C}X),$ is vertical part (resp.
horizontal part) of $\phi X$.

From (\ref{metric}), (\ref{TAN}) and (\ref{NOR}) we obtain%
\begin{equation}
g_{M}(\psi U,V)=-g_{M}(U,\psi V)  \label{ANTT}
\end{equation}

and%
\begin{equation}
g_{M}(\omega U,Y)=-g_{M}(U,\mathcal{B}Y).  \label{ANTN}
\end{equation}%
for any $U,V$ $\in \Gamma (\ker F_{\ast })$ and $Y\in \Gamma ((\ker F_{\ast
})^{\bot })$.

Using (\ref{1}), (\ref{3}) and (\ref{xzeta}) we obtain%
\begin{equation}
\mathcal{T}_{U}\xi =0,  \label{CONNECTION}
\end{equation}%
for any $U\in \Gamma (\ker F_{\ast }).$

From (\ref{fi}), (\ref{TAN}) and (\ref{NOR}) we can easily obtain following
lemma.

\begin{lemma}
\label{LEM}Let $F$ be an slant Riemannian submersion from a cosymplectic
manifold $M(\phi ,\xi ,\eta ,g_{M})$ to a Riemannian manifold $(N,g_{N})$.
Then we have%
\begin{eqnarray*}
-X &=&\omega \mathcal{B}X+\mathcal{C}^{2}X,\text{ \ } \\
0 &=&\psi \mathcal{B}X+\mathcal{BC}X, \\
\phi ^{2}U &=&\psi ^{2}U+\mathcal{B}\omega U,\text{ } \\
0 &=&\omega \psi U+\mathcal{C}\omega U, \\
g_{M}(\mathcal{C}X,\phi U) &=&-g_{M}(\mathcal{B}X,\psi U)
\end{eqnarray*}%
for any $X\in \Gamma ((\ker F_{\ast })^{\bot })$ and $V\in \Gamma ((\ker
F_{\ast }))$.
\end{lemma}

\begin{proposition}
(\cite{KUPMUR}) \label{PROP.}Let $F$ be a Riemannian submersion from almost
contact manifold onto a Riemannian manifold. If dim($\ker F_{\ast })=2$ and $%
\xi $\ is vertical then fibers are anti-invariant.
\end{proposition}

As the proof of the following proposition is similar to slant submanifolds
(see \cite{alfonso}) we don't give its proof.

\begin{proposition}
\label{d}Let $F$ be a Riemannian submersion from a cosymplectic manifold $%
M(\phi ,\xi ,\eta ,g_{M})$ onto a Riemannian manifold $(N,g_{N})$ such that $%
\xi \in \Gamma (\ker F_{\ast })$. Then $F$ is anti-invariant submersion if
and only if $D$ is integrable, where $D=\ker F_{\ast }-\{\xi \}$.
\end{proposition}

\begin{theorem}
\label{tu}Let $M(\phi ,\xi ,\eta ,g_{M})$ be a cosymplectic manifold \ of
dimension $2m+1$ and $(N,g_{N})$ is a Riemannian manifold of dimension $n$.
Let $F:M(\phi ,\xi ,\eta ,g_{M})\rightarrow $ $(N,g_{N})$ be a slant
Riemannian submersion. Then the fibers are not proper totally umbilical.
\end{theorem}

\begin{proof}
If the fibers are proper totally umbilical, then we have $\mathcal{T}%
_{U}V=g_{M}(U,V)H$ for any vertical vector fields $U,V$ where $H$ is the
mean curvature vector field of any fibre. Since $\mathcal{T}_{\xi }\xi $ $=0$%
, we have $H=0$, which shows that fibres are minimal. Hence the fibers are
totally geodesic. This completes proof of the Theorem.
\end{proof}

By (\ref{1}), (\ref{2}), (\ref{TAN}) and (\ref{NOR}) we have%
\begin{equation}
(\nabla _{U}\omega )V=\mathcal{C}T_{U}V-\mathcal{T}_{U}\psi V,  \label{W}
\end{equation}%
\begin{equation}
(\nabla _{U}\psi )V=\mathcal{BT}_{U}V-\mathcal{T}_{U}\omega V,  \label{F}
\end{equation}%
where%
\begin{eqnarray*}
(\nabla _{U}\omega )V &=&\mathcal{H}\nabla _{U}\omega V-\omega \hat{\nabla}%
_{U}V, \\
(\nabla _{U}\psi )V &=&\hat{\nabla}_{U}\psi V-\psi \hat{\nabla}_{U}V
\end{eqnarray*}%
for $U,V\in \Gamma (\ker F_{\ast }).$

We will give a characterization theorem for slant submersions of a
cosymplectic manifold. Since the proof of the following theorem is quite
similar to (Theorem 3 of \cite{KUPMUR}) , we don't give its proof.

\begin{theorem}
\label{pisi}Let $F$ be a Riemannian submersion from a cosymplectic manifold $%
M(\phi ,\xi ,\eta ,g_{M})$ onto a Riemannian manifold $(N,g_{N})$ such that $%
\xi \in \Gamma (\ker F_{\ast })$. Then, $F$ is a slant Riemannian submersion
if and only if there exist a constant $\lambda \in \lbrack 0,1]$ such that 
\begin{equation}
\psi ^{2}=-\lambda (I-\eta \otimes \xi ).  \label{SLANT}
\end{equation}%
Furthermore, in such case, if $\theta $ is the slant angle of $F$, it
satisfies that $\lambda =\cos ^{2}\theta .$
\end{theorem}

By using (\ref{metric}), (\ref{TAN}), (\ref{ANTT}) and (\ref{SLANT}) we have
following Lemma.

\begin{lemma}
\label{pisi2}Let $F$ be a slant Riemannian submersion from a cosymplectic
manifold $M(\phi ,\xi ,\eta ,g_{M})$ onto a Riemannian manifold $(N,g_{N})$
with slant angle $\theta .$Then the following relations are valid 
\begin{equation}
g_{M}(\psi U,\psi V)=\cos ^{2}\theta (g_{M}(U,V)-\eta (U)\eta (V)),
\label{COS5A}
\end{equation}%
\begin{equation}
g_{M}(\omega U,\omega V)=\sin ^{2}\theta (g_{M}(U,V)-\eta (U)\eta (V))
\label{COS6}
\end{equation}%
for any $U,V\in \Gamma (\ker F_{\ast }).$
\end{lemma}

We denote the complementary orthogonal distribution to $\omega (\ker F_{\ast
})$ in $(\ker F_{\ast })^{\bot }$ by $\mu .$ Then we have%
\begin{equation}
(\ker F_{\ast })^{\bot }=\omega (\ker F_{\ast })\oplus \mu .  \label{A1}
\end{equation}

\begin{lemma}
\label{ir}Let $F$ be a proper slant Riemannian submersion from a
cosymplectic manifold $M(\phi ,\xi ,\eta ,g_{M})$ onto a Riemannian manifold 
$(N,g_{N})$ then $\mu $ is an invariant distribution of $(\ker F_{\ast
})^{\bot },$ under the endomorphism $\phi $.
\end{lemma}

\begin{proof}
For $X\in \Gamma (\mu ),$ from (\ref{metric}) and (\ref{TAN}), we obtain%
\begin{eqnarray*}
g_{M}(\phi X,\omega V) &=&g_{M}(\phi X,\phi V)-g_{M}(\phi X,\psi V) \\
&=&g_{M}(X,V)-\eta (X)\eta (V)-g_{M}(\phi X,\psi V) \\
&=&+g_{M}(X,\phi \psi V).
\end{eqnarray*}%
Using (\ref{SLANT}) and (\ref{A1}) we have%
\begin{eqnarray*}
g_{M}(\phi X,\omega V) &=&-\cos ^{2}\theta g_{M}(X,V-\eta (V)\xi ) \\
&=&g_{M}(X,\omega \psi V) \\
&=&0.
\end{eqnarray*}%
In a similar way, we have $g_{M}(\phi X,U)=-g_{M}(X,\phi U)=0$ due to $\phi
U\in \Gamma ((\ker F_{\ast })\oplus \omega (\ker F_{\ast }))$ for $X\in
\Gamma (\mu )$ and $U\in \Gamma (\ker F_{\ast }).$Thus the proof of the
lemma is completed.
\end{proof}

By help (\ref{COS6}), we can give following

\begin{corollary}
\label{teta}Let $F$ be a proper slant Riemannian submersion from a
cosymplectic manifold $M^{2m+1}(\phi ,\xi ,\eta ,g_{M})$ onto a Riemannian
manifold $(N^{n},g_{N}).$Let%
\begin{equation*}
\left\{ e_{1},e_{2},...e_{2m-n},\xi \right\}
\end{equation*}%
be a local orthonormal basis of $(\ker F_{\ast }),$ then $\left\{ \csc
\theta we_{1},\csc \theta we_{2},...,\csc \theta we_{2m-n}\right\} $ is a
local orthonormal basis of $\omega (\ker F_{\ast }).$
\end{corollary}

By using (\ref{A1}) and Corollary \ref{teta} one can easily prove the
following Proposition.

\begin{proposition}
\label{mu}Let $F$ be a proper slant Riemannian submersion from a
cosymplectic manifold $M^{2m+1}(\phi ,\xi ,\eta ,g_{M})$ onto a Riemannian
manifold $(N^{n},g_{N}).$Then dim$(\mu )=2(n-m).$ If $\mu =\left\{ 0\right\}
,$then $n=m.$
\end{proposition}

By (\ref{ANTT}) and (\ref{COS5A}) we have

\begin{lemma}
\label{tetaapisi}Let $F$ be a proper slant Riemannian submersion from a
cosymplectic manifold $M^{2m+1}(\phi ,\xi ,\eta ,g_{M})$ onto a Riemannian
manifold $(N^{n},g_{N}).$ If $e_{1},e_{2},...e_{k},\xi $ are orthogonal unit
vector fields in $(\ker F_{\ast }),$then%
\begin{equation*}
\left\{ e_{1},\sec \theta \psi e_{1},e_{2},\sec \theta \psi
e_{2},...e_{k},\sec \theta \psi e_{k},\xi \right\}
\end{equation*}%
is a local orthonormal basis of $(\ker F_{\ast }).$ Moreover dim$(\ker
F_{\ast })=2m-n+1=2k+1$ and $\dim N=n=2(m-k).$
\end{lemma}

\begin{lemma}
\label{TT}Let $F$ be a slant Riemannian submersion from a cosymplectic
manifold $M(\phi ,\xi ,\eta ,g_{M})$ onto a Riemannian manifold $(N,g_{N})$%
.\ If $\omega $ is parallel then we have%
\begin{equation}
\mathcal{T}_{\psi U}\psi U=-\cos ^{2}\theta \mathcal{T}_{U}U.  \label{SEC1}
\end{equation}

\begin{proof}
If $\omega $ is parallel, from (\ref{W}), we obtain $\mathcal{C}T_{U}V=%
\mathcal{T}_{U}\psi V$ for $U,V\in \Gamma (\ker F_{\ast }).$We interchange $%
U $ and $V$ and use (\ref{TUW}) we get 
\begin{equation*}
\mathcal{T}_{U}\psi V=\mathcal{T}_{V}\psi U.
\end{equation*}%
Substituting $V$ by $\psi U$ in the above equation and then using Theorem 2
we get the required formula.
\end{proof}
\end{lemma}

We give a sufficient condition for a slant Riemannian submersion to be
harmonic as an analogue of a slant Riemannian submersion from a cosymplectic
manifold onto a Riemannian manifold in \cite{SAHIN1}.

\begin{theorem}
\label{HAR}Let $F$ be a slant Riemannian submersion from a cosymplectic
manifold $M(\phi ,\xi ,\eta ,g_{M})$ onto a Riemannian manifold $(N,g_{N})$
\ If $\omega $ is parallel then $F$ is a harmonic map.
\end{theorem}

\begin{proof}
The proof is similar to the proof of (Theorem 4 of \cite{KUPMUR} and \cite%
{SAHIN1}).
\end{proof}

Now putting $Q=\psi ^{2}$,we define $\nabla Q$ by%
\begin{equation}
(\nabla _{U}Q)V=\mathcal{V}\nabla _{U}QV-Q\hat{\nabla}_{U}V  \label{VERQ}
\end{equation}%
for any $U,V\in \Gamma (\ker F_{\ast }).$

\begin{theorem}
\label{QVER}Let $F$ be a slant Riemannian submersion from a cosymplectic
manifold $M(\phi ,\xi ,\eta ,g_{M})$ onto a Riemannian manifold $(N,g_{N})$%
.\ Then $\nabla Q=0.$
\end{theorem}

\begin{proof}
From (\ref{SLANT}), 
\begin{equation}
Q\hat{\nabla}_{U}V=-\cos ^{2}\theta (\hat{\nabla}_{U}V-\eta (\hat{\nabla}%
_{U}V)\xi )  \label{Q1}
\end{equation}%
for each $U,V\in \Gamma (\ker F_{\ast }),$ where $\theta $ is slant angle.
On the other hand,%
\begin{equation}
\mathcal{V}(\nabla _{U}QV)=-\cos ^{2}\theta (\hat{\nabla}_{U}V-\eta (\hat{%
\nabla}_{U}V)\xi .  \label{Q2}
\end{equation}%
Combining (\ref{Q1}) and (\ref{Q2}), we obtain $(\nabla _{U}Q)V=0.$ This
completes the proof.
\end{proof}

We now investigate the geometry of leaves of ($\ker F_{\ast })^{\perp }$ and 
$\ker F_{\ast }.$

\begin{theorem}
\label{HAB}Let $F$ be a slant Riemannian submersion from a cosymplectic
manifold $M(\phi ,\xi ,\eta ,g_{M})$ onto a Riemannian manifold $(N,g_{N}).$%
Then the distribution $(\ker F_{\ast })^{\bot }$ defines a totally geodesic
foliation on $M$ if and only if 
\begin{equation*}
g_{M}(\mathcal{H}\nabla _{X}Y,\omega \psi U)=g_{M}(\mathcal{A}_{X}\mathcal{B}%
Y,\omega U)+g_{M}(\mathcal{H}\nabla _{X}\mathcal{C}Y,\omega U)
\end{equation*}%
for any $X,Y\in \Gamma ((\ker F_{\ast })^{\bot })$ and $U\in \Gamma (\ker
F_{\ast }).$
\end{theorem}

\begin{proof}
From (\ref{nambla}) and (\ref{TAN}) we have%
\begin{eqnarray}
g_{M}(\nabla _{X}Y,U) &=&-g_{M}(\phi \nabla _{X}\phi Y,U)+g_{M}(\nabla
_{X}Y,\xi )\eta (U)  \label{TOTA} \\
&=&g_{M}(\nabla _{X}\phi Y,\phi U)+g_{M}(\nabla _{X}Y,\xi )\eta (U)  \notag
\\
&=&g_{M}(\phi \nabla _{X}Y,\psi U)+g_{M}(\nabla _{X}\phi Y,\omega
U)+g_{M}(\nabla _{X}Y,\xi )\eta (U).  \notag
\end{eqnarray}%
for any $X,Y\in \Gamma ((\ker F_{\ast })^{\bot })$ and $\ U\in \Gamma (\ker
F_{\ast })$.

Using (\ref{nambla}) and (\ref{TAN}) in (\ref{TOTA}), we obtain 
\begin{eqnarray}
g_{M}(\nabla _{X}Y,U) &=&-g_{M}(\nabla _{X}Y,\psi ^{2}U)-g_{M}(\nabla
_{X}Y,\omega \psi U)  \label{TOT3} \\
&&+g_{M}(\nabla _{X}Y,\xi )\eta (U)+g_{M}(\nabla _{X}\phi Y,\omega U). 
\notag
\end{eqnarray}%
By (\ref{NOR}) and (\ref{SLANT}) we have%
\begin{eqnarray}
g_{M}(\nabla _{X}Y,U) &=&\cos ^{2}\theta g_{M}(\nabla _{X}Y,U)-\cos
^{2}\theta \eta (U)\eta (\nabla _{X}Y)  \label{TOT4} \\
&&-g_{M}(\nabla _{X}Y,\omega \psi U)+g_{M}(\nabla _{X}Y,\xi )\eta (U)  \notag
\\
&&+g_{M}(\nabla _{X}\mathcal{B}Y,\omega U)+g_{M}(\nabla _{X}\mathcal{C}%
Y,\omega U).  \notag
\end{eqnarray}%
Using (\ref{3}), (\ref{4}) and (\ref{xzeta}) in the last equation we obtain%
\begin{equation*}
g_{M}(\mathcal{H}\nabla _{X}Y,\omega \psi U)=g_{M}(\mathcal{A}_{X}\mathcal{B}%
Y,\omega U)+g_{M}(\mathcal{H}\nabla _{X}\mathcal{C}Y,\omega U)
\end{equation*}%
which prove the theorem.
\end{proof}

\begin{theorem}
\label{HTH}Let $F$ be a slant Riemannian submersion from a cosymplectic
manifold $M(\phi ,\xi ,\eta ,g_{M})$ onto a Riemannian manifold $(N,g_{N}).$%
Then the distribution $(\ker F_{\ast })$ defines a totally geodesic
foliation on $M$ if and only if 
\begin{equation*}
g_{M}(\mathcal{H}\nabla _{U}\omega \psi V,X)=g_{M}(\mathcal{T}_{U}\omega V,%
\mathcal{B}X)+g_{M}(\mathcal{H}\nabla _{U}\omega V,\mathcal{C}X)
\end{equation*}%
for any $U,V\in \Gamma (\ker F_{\ast })$ and $X\in \Gamma ((\ker F_{\ast
})^{\bot }).$
\end{theorem}

\begin{proof}
From (\ref{nambla}) and (\ref{TAN}) we have%
\begin{eqnarray}
g_{M}(\nabla _{U}V,X) &=&-g_{M}(\phi \nabla _{U}\phi V,X)+g_{M}(\nabla
_{U}V,\xi )\eta (X)  \label{TOTA1} \\
&=&-g_{M}(\phi \nabla _{U}\psi V,X)-g_{M}(\phi \nabla _{U}\omega V,X)  \notag
\\
&=&-g_{M}(\phi \nabla _{U}\psi V,X)+g_{M}(\nabla _{U}\omega V,\phi X). 
\notag
\end{eqnarray}%
for any $U,V\in \Gamma (\ker F_{\ast })$ and $X\in \Gamma ((\ker F_{\ast
})^{\bot }).$

Using (\ref{nambla}), (\ref{TAN}) and (\ref{NOR}) in (\ref{TOTA1}), we
obtain 
\begin{eqnarray}
g_{M}(\nabla _{U}V,X) &=&-g_{M}(\nabla _{U}\psi ^{2}V,X)-g_{M}(\nabla
_{U}\omega \psi V,X)  \label{TOT3,5} \\
&&+g_{M}(\nabla _{U}\omega V,\mathcal{B}X)+g_{M}(\nabla _{U}\omega V,%
\mathcal{C}X).  \notag
\end{eqnarray}%
By (\ref{SLANT}),(\ref{3}), (\ref{4}) and (\ref{xzeta}) in the last equation
we obtain requested relation.
\end{proof}

From Teorem 4 and Teorem 5 we have the following corollary.

\begin{corollary}
\label{CHTH}Let $F$ be a slant Riemannian submersion from a cosymplectic
manifold $M(\phi ,\xi ,\eta ,g_{M})$ onto a Riemannian manifold $(N,g_{N}).$%
Then $M$ is a locally product Riemannian manifold if and only if 
\begin{equation*}
g_{M}(\mathcal{H}\nabla _{U}\omega \psi V,X)=g_{M}(\mathcal{T}_{U}\omega V,%
\mathcal{B}X)+g_{M}(\mathcal{H}\nabla _{U}\omega V,\mathcal{C}X)
\end{equation*}%
and%
\begin{equation*}
g_{M}(\mathcal{H}\nabla _{X}Y,\omega \psi U)=g_{M}(\mathcal{A}_{X}\mathcal{B}%
Y,\omega U)+g_{M}(\mathcal{H}\nabla _{X}\mathcal{C}Y,\omega U)
\end{equation*}%
for any $U,V\in \Gamma (\ker F_{\ast })$ and $X\in \Gamma ((\ker F_{\ast
})^{\bot }).$
\end{corollary}

For a slant Riemannian submersion we have the following Theorem.

\begin{theorem}
\label{TG}Let $F$ be a slant Riemannian submersion from a cosymplectic
manifold $M(\phi ,\xi ,\eta ,g_{M})$ onto a Riemannian manifold $(N,g_{N}).$%
Then $F$ is totally geodesic if and only if%
\begin{equation*}
g_{M}(\mathcal{H}\nabla _{U}\omega \psi V,X)=g_{M}(\mathcal{T}_{U}\omega V,%
\mathcal{B}X)+g_{M}(\mathcal{H}\nabla _{U}\omega V,\mathcal{C}X)
\end{equation*}%
and%
\begin{equation*}
g_{M}(\mathcal{H}\nabla _{Z}\omega \psi U,X)=g_{M}(\mathcal{A}_{Z}\omega U,%
\mathcal{B}X)+g_{M}(\mathcal{H}\nabla _{Z}\omega U,\mathcal{C}X)
\end{equation*}%
for $U,V\in \Gamma (\ker F_{\ast })$ and $X,Y\in \Gamma ((\ker F_{\ast
})^{\bot }).$
\end{theorem}

\begin{proof}
First of all, we recall that Riemannian submersion satisfies 
\begin{equation*}
(\nabla F_{\ast })(X,Y)=0
\end{equation*}%
for $X,Y\in \Gamma ((\ker F_{\ast })^{\bot }).$ We will prove that $(\nabla
F_{\ast })(U,V)=0$ and $(\nabla F_{\ast })(X,U)=0$ for $U,V\in \Gamma (\ker
F_{\ast })$ and $X\in \Gamma ((\ker F_{\ast })^{\bot }).$Since $F$ is a
Riemannian submersion, from (\ref{fi}) and (\ref{nambla}) we get%
\begin{equation}
\nabla _{U}V=-\phi \nabla _{U}\phi V+\eta (\nabla _{U}V)\xi .  \label{irem}
\end{equation}%
From (\ref{TAN}) and (\ref{irem}) we obtain%
\begin{equation*}
g_{N}((\nabla F_{\ast })(U,V),F_{\ast }X)=g_{M}(\nabla _{U}\phi \psi
V,X)-g_{M}(\nabla _{U}\omega V,\phi X).
\end{equation*}%
Using again (\ref{TAN}) and (\ref{NOR}) we get%
\begin{eqnarray*}
g_{N}((\nabla F_{\ast })(U,V),F_{\ast }X) &=&g_{M}(\nabla _{U}\psi
^{2}V,X)+g_{M}(\nabla _{U}\omega \psi V,X) \\
&&-g_{M}(\nabla _{U}\omega V,\mathcal{B}X)-g_{M}(\nabla _{U}\omega V,%
\mathcal{C}X).
\end{eqnarray*}%
Theorem 2, (\ref{1}) and (\ref{2}) imply that%
\begin{eqnarray*}
g_{N}((\nabla F_{\ast })(U,V),F_{\ast }X) &=&-\cos ^{2}\theta g_{M}(\nabla
_{U}V,X)+g_{M}(\nabla _{U}\omega \psi V,X) \\
&&-g_{M}(\mathcal{T}_{U}\omega V,\mathcal{B}X)-g_{M}(\mathcal{H}\nabla
_{U}\omega V,\mathcal{C}X).
\end{eqnarray*}%
After some calculation we have%
\begin{eqnarray}
\sin ^{2}\theta g_{N}((\nabla F_{\ast })(U,V),F_{\ast }X) &=&g_{M}(\nabla
_{U}\omega \psi V,X)-g_{M}(\mathcal{T}_{U}\omega V,\mathcal{B}X)
\label{irem2} \\
&&-g_{M}(\mathcal{H}\nabla _{U}\omega V,\mathcal{C}X).  \notag
\end{eqnarray}%
In a similar way, we get%
\begin{eqnarray}
\sin ^{2}\theta g_{N}((\nabla F_{\ast })(Z,U),F_{\ast }X) &=&g_{M}(\nabla
_{Z}\omega \psi U,X)-g_{M}(\mathcal{A}_{Z}\omega U,\mathcal{B}X)
\label{irem3} \\
&&--g_{M}(\mathcal{H}\nabla _{Z}\omega U,\mathcal{C}X).  \notag
\end{eqnarray}%
Then proof follows from (\ref{irem2}) and (\ref{irem3}).
\end{proof}

Now we establish a sharp inequality between squared mean curvature $\Vert
H\Vert ^{2}$ and the scalar curvature $\hat{\tau}$ of fibre through $p\in
M^{5}(c).$

\begin{theorem}
\label{VERTICAL}Let $F$ be a proper slant Riemannian submersion from a
cosymplectic space form $M^{5}(c)$ onto a Riemannian manifold $(N^{2},g_{N})$%
. Then, we have%
\begin{equation}
\Vert H\Vert ^{2}\geq \frac{8}{9}(\hat{\tau}-\frac{c}{4}(1+3\cos ^{2}\theta
)),  \label{h2inequ}
\end{equation}%
where $H$ denotes mean curvature of fibers. Moreover, the equality sign of (%
\ref{h2inequ}) holds at a point $p$ of a fiber if and only if with respect
to some suitable slant orthonormal basis $\left\{ e_{1},e_{2}=\sec \theta
\psi e_{1},e_{3}=\xi ,e_{4}=\csc \theta we_{1},e_{5}=\csc \theta
we_{2}\right\} $ at $p$, 
\begin{equation*}
T_{11}^{4}=3T_{22}^{4},\text{ \ \ }T_{12}^{4}=0\text{ and }T_{11}^{5}=0.
\end{equation*}%
where $T_{ij}^{\alpha }=g(\mathcal{T}(e_{i},e_{j}),e_{\alpha })$ for $1\leq
i,j\leq 3$ and $\alpha =4,5.$
\end{theorem}

\begin{proof}
By Corollary \ref{teta} and Lemma \ref{tetaapisi}, we construct a slant
orthonormal basis $\left\{ e_{1},e_{2},e_{3},e_{4},e_{5}\right\} $ defined
by 
\begin{equation}
e_{1},e_{2}=\sec \theta \psi e_{1},e_{3}=\xi ,e_{4}=\csc \theta
we_{1},e_{5}=\csc \theta we_{2}.  \label{SLANT ORT}
\end{equation}%
Let $\hat{\tau}$ be scalar curvature of a fibre $F^{-1}(q)$. We choose
arbitrary point $p$ of the fibre $F^{-1}(q).$We obtain%
\begin{equation}
\hat{\tau}(p)=\hat{K}(e_{1}\wedge e_{2})+\hat{K}(e_{1}\wedge e_{3})+\hat{K}%
(e_{2}\wedge e_{3}).  \label{mico}
\end{equation}%
By (\ref{4D}), (\ref{4E}) and (\ref{CURVATURE}) we get 
\begin{equation}
\hat{K}(e_{1}\wedge e_{2})=\frac{c}{4}(1+3\cos ^{2}\theta
)+T_{11}^{4}T_{22}^{4}+T_{11}^{5}T_{22}^{5}-(T_{12}^{4})^{2}-(T_{12}^{5})^{2},
\label{IREM}
\end{equation}%
where $T_{ij}^{\alpha }=g(\mathcal{T}(e_{i},e_{j}),e_{\alpha })$ for $1\leq
i,j\leq 3$ and $\alpha =4,5.$ Using Theorem \ref{pisi} \ and the relation (%
\ref{COS6}) one has 
\begin{equation}
\psi e_{2}=-\cos \theta e_{1}\text{ \ \ and \ \ }\omega e_{2}=\sin \theta
e_{5}.  \label{OGUZ}
\end{equation}%
From (\ref{F}) we have%
\begin{equation*}
g((\hat{\nabla}_{e_{2}}\psi )e_{2},e_{1})=g(\mathcal{BT}%
_{e_{2}}e_{2},e_{1})-g(\mathcal{T}_{e_{2}}\omega e_{2},e_{1}).
\end{equation*}%
Using (\ref{OGUZ}) in the last relation we obtain 
\begin{equation}
\sin \theta \lbrack g(\mathcal{T}_{e_{2}}e_{2},e_{4})-g(\mathcal{T}%
_{e_{2}}e_{1},e_{5})]=0.  \label{bilal}
\end{equation}%
Since our submersion is proper, the equation (\ref{bilal}) implies%
\begin{equation*}
T_{22}^{4}=T_{12}^{5}.
\end{equation*}%
Now we chose the unit normal vector $e_{4}\in \Gamma (\ker (F_{\ast
}))^{\perp }$ parallel to the mean curvature vector $H(p)$ of fibre. Then we
have 
\begin{equation*}
\Vert H(p)\Vert ^{2}=\frac{1}{9}(T_{11}^{4}+T_{22}^{4})^{2},\text{ \ \ }%
T_{11}^{5}+T_{22}^{5}=0.
\end{equation*}%
So the relation (\ref{IREM}) becomes 
\begin{equation}
\hat{K}(e_{1}\wedge e_{2})=\frac{c}{4}(1+3\cos ^{2}\theta
)+T_{11}^{4}T_{22}^{4}-(T_{11}^{5})^{2}-(T_{12}^{4})^{2}-(T_{22}^{4})^{2}.
\label{metehan}
\end{equation}%
From the trivial inequality $(\mu -3\lambda )^{2}\geq 0$ , one has $(\mu
+\lambda )^{2}\geq 8(\lambda \mu -\lambda ^{2}).$ Putting $\mu =T_{11}^{4}$
and $\lambda =T_{22}^{4}$ in the last inequality we find 
\begin{equation}
\Vert H\Vert ^{2}\geq \frac{8}{9}[\hat{K}(e_{1}\wedge e_{2})-\frac{c}{4}%
(1+3\cos ^{2}\theta )].  \label{kadir}
\end{equation}%
Using (\ref{4E}) we get 
\begin{equation*}
\hat{K}(e_{1}\wedge e_{3})=\hat{K}(e_{2}\wedge e_{3})=0.
\end{equation*}%
By (\ref{4E}), (\ref{mico}) and the last relation \ we get required
inequality. Moreover, the equality sign of (\ref{h2inequ}) holds at a point $%
p$ of a fiber if and only if $T_{11}^{4}=3T_{22}^{4},$ $T_{12}^{4}=0$ and $%
T_{11}^{5}=0.$
\end{proof}

\subsection{Slant Riemannian submersions admitting horizontal structure
vector field}

In this subsection, we will investigate the properties of slant Riemannian
submersions from cosymplectic manifolds onto a Riemannian manifold such that
characteristic vector field $\xi $\textit{\ is horizontal}.

By (\ref{fi}),(\ref{metric}) and (\ref{nambla}) we have 
\begin{equation}
\phi ^{2}U=-U\text{ \ , \ }(\nabla _{U}\phi )V=0\text{ and }g(\phi U,\phi
V)=g(U,V),\forall U,V\in \Gamma (\ker F_{\ast }).  \label{HHH}
\end{equation}%
Fom (\ref{2}), (\ref{4b}), (\ref{xzeta}) and (\ref{HHH}) we obtain 
\begin{equation}
i)\mathcal{T}_{U}\xi =0\text{ \ }ii)\eta (\nabla _{U}V)=0\text{ and }iii)%
\mathcal{A}_{X}\xi =0  \label{HH1}
\end{equation}%
and%
\begin{equation}
\nabla _{U}V=-\phi \nabla _{U}\phi V  \label{HH2}
\end{equation}%
for any $U,V\in \Gamma (\ker F_{\ast })$ and $X(\Gamma (\ker F_{\ast
})^{\perp }).$ By the following same steps in (see: \cite{alfonso}, \cite%
{KUPMUR}) we can prove the following characterization theorem:

\begin{theorem}
\label{HOR1}Let $F$ be a Riemannian submersion from a cosymplectic manifold $%
M(\phi ,\xi ,\eta ,g_{M})$ onto a Riemannian manifold $(N,g_{N})$ such that $%
\xi \in (\Gamma (\ker F_{\ast })^{\perp })$. Then, $F$ is a slant Riemannian
submersion if and only if there exist a constant $\lambda \in \lbrack 0,1]$
such that 
\begin{equation}
\psi ^{2}=-\lambda I.  \label{H2}
\end{equation}%
Furthermore, in such case, if $\theta $ is the slant angle of $F$, then $%
\lambda =\cos ^{2}\theta .$
\end{theorem}

By Thorem \ref{HOR1}, we get following

\begin{lemma}
\label{HOR2} Let $F$ be a Riemannian submersion from a cosymplectic manifold 
$M(\phi ,\xi ,\eta ,g_{M})$ onto a Riemannian manifold $(N,g_{N})$ with
slant angle $\theta .$ Then we have the following relations: 
\begin{equation}
g_{M}(\psi U,\psi V)=\cos ^{2}\theta g_{M}(U,V),  \label{H3}
\end{equation}%
\begin{equation}
g_{M}(\omega U,\omega V)=\sin ^{2}\theta g_{M}(U,V)  \label{H4}
\end{equation}%
for any $U,V\in \Gamma (\ker F_{\ast }).$
\end{lemma}

\begin{remark}
\label{HOR3A}Let $F$ be a Riemannian submersion from a cosymplectic manifold 
$M(\phi ,\xi ,\eta ,g_{M})$ onto a Riemannian manifold $(N,g_{N})$ with
slant angle $\theta $ and $\xi \in (\Gamma (\ker F_{\ast })^{\perp }).$ Then
there is a distribution $\mathcal{D}$\emph{\ }$\subset (\Gamma (\ker F_{\ast
})^{\perp })$ such that $\Gamma (\ker F_{\ast })^{\perp }=\omega (\ker
F_{\ast })\oplus \mathcal{D}\oplus \{\xi \}.$
\end{remark}

In the similar argumentation of\ Lemma \ref{ir}, we get

\begin{lemma}
\label{HOR3B}Let $F$ be a Riemannian submersion from a cosymplectic manifold 
$M(\phi ,\xi ,\eta ,g_{M})$ onto a Riemannian manifold $(N,g_{N})$ with
slant angle $\theta $ and $\xi \in (\Gamma (\ker F_{\ast })^{\perp }).$ Then
the distribution $\mathcal{D}$ is invariant under $\phi .$
\end{lemma}

Using same arguments in proof of Theorem \ref{QVER} and \ref{VERQ} we find

\begin{theorem}
\label{HOR3}Let $F$ be a slant Riemannian submersion from a cosymplectic
manifold $M(\phi ,\xi ,\eta ,g_{M})$ onto a Riemannian manifold $(N,g_{N})$
with slant angle $\theta $.\ Then $\nabla Q=0.$
\end{theorem}

By virtue of (\ref{H4}), we can state

\begin{corollary}
\label{HOR3C}Let $F$ be a Riemannian submersion from a cosymplectic manifold 
$M(\phi ,\xi ,\eta ,g_{M})$ onto a Riemannian manifold $(N,g_{N})$ with
slant angle $\theta $ and $\xi \in (\Gamma (\ker F_{\ast })^{\perp }).$ If $%
\left\{ e_{1},e_{2},...e_{2m-n+1}\right\} $ be a local orthonormal basis of $%
(\ker F_{\ast }),$ then $\left\{ \csc \theta we_{1},\csc \theta
we_{2},...,\csc \theta we_{2m-n+1}\right\} $ is a local orthonormal basis of 
$\omega (\ker F_{\ast }).$\ Moreover $dim(\mathcal{D)=}2(n-m-1).$
\end{corollary}

By using (\ref{H3}), we can prove

\begin{lemma}
\label{HOR4}Let $F$ be a proper slant Riemannian submersion from a
cosymplectic manifold $M^{2m+1}(\phi ,\xi ,\eta ,g_{M})$ onto a Riemannian
manifold $(N^{n},g_{N})$ with slant angle $\theta $ and $\xi \in (\Gamma
(\ker F_{\ast })^{\perp }).$ If $e_{1},e_{2},...e_{k}$ are orthogonal unit
vector fields in $(\ker F_{\ast }),$then%
\begin{equation*}
\left\{ e_{1},\sec \theta \psi e_{1},e_{2},\sec \theta \psi
e_{2},...e_{k},\sec \theta \psi e_{k}\right\}
\end{equation*}%
is a local orthonormal basis of $(\ker F_{\ast }).$ Moreover $\dim (\ker
F_{\ast })=2m-n+1=2k$ and $\dim N=n=2(m-k)+1.$
\end{lemma}

\begin{remark}
Lemma \ref{HOR4} state that dimension of $(\ker F_{\ast })$ and dimension of 
$N$ should be even and odd respectively.
\end{remark}

\begin{lemma}
\label{HOR5}Let $F$ be a proper slant Riemannian submersion from a
cosymplectic manifold $M^{2m+1}(\phi ,\xi ,\eta ,g_{M})$ onto a Riemannian
manifold $(N^{n},g_{N})$ with slant angle $\theta $ and $\xi \in (\Gamma
(\ker F_{\ast })^{\perp }).$ If $\omega $ is parallel then we have%
\begin{equation}
\mathcal{T}_{\psi U}\psi U=-\cos ^{2}\theta \mathcal{T}_{U}U  \label{H5}
\end{equation}
\end{lemma}

\begin{proof}
This proof works like the Kaehlerian case in (\cite{SAHIN1}).
\end{proof}

The following result gives sufficient condition to obtain the proper
harmonic slant submersion.

\begin{theorem}
\label{HOR6}Let $F$ be a slant Riemannian submersion from a cosymplectic
manifold $M(\phi ,\xi ,\eta ,g_{M})$ onto a Riemannian manifold $(N,g_{N})$
with slant angle $\theta $ and $\xi \in (\Gamma (\ker F_{\ast })^{\perp }).$
\ If $\omega $ is parallel then $F$ is a harmonic map.

\begin{proof}
By help of (\ref{5a}) we know that 
\begin{equation}
(\nabla F_{\ast })(X,Y)=0,  \label{H6}
\end{equation}%
for any horizontal vetor fields $X,Y.$ From (\ref{5}),(\ref{6}), (\ref{H6})
and Lemma \ref{HOR4} we $\ $get 
\begin{eqnarray}
\tau (F) &=&-\dsum\limits_{i=1}^{m-(\frac{n+1}{2})}[(\nabla F_{\ast
})(e_{i},e_{i})+(\nabla F_{\ast })(\sec \theta \psi e_{i},\sec \theta \psi
e_{i})]  \notag \\
&=&-\dsum\limits_{i=1}^{m-(\frac{n+1}{2})}[F_{\ast }(\nabla
_{e_{i}}e_{i})+\sec ^{2}\theta F_{\ast }(\nabla _{e_{i}}e_{i})],  \label{H7}
\end{eqnarray}%
where $\left\{ e_{1},\sec \theta \psi e_{1},e_{2},\sec \theta \psi
e_{2},...e_{m-(\frac{n+1}{2})},\sec \theta \psi e_{m-(\frac{n+1}{2}%
)}\right\} $is an orthonormal basis of $(\ker F_{\ast }).$ By applying (\ref%
{1}) in (\ref{H7}), we obtain%
\begin{equation*}
\tau (F)=-\dsum\limits_{i=1}^{m-(\frac{n+1}{2})}F_{\ast }(\mathcal{T}%
_{e_{i}}e_{i}+\sec ^{2}\theta \mathcal{T}_{\psi e_{i}}\psi e_{i}).
\end{equation*}%
Then using Lemma \ref{HOR5} we have 
\begin{equation*}
\tau (F)=-\dsum\limits_{i=1}^{m-(\frac{n+1}{2})}F_{\ast }(\mathcal{T}%
_{e_{i}}e_{i}-\mathcal{T}_{e_{i}}e_{i})=0
\end{equation*}%
which says that $F$ is a harmonic map.
\end{proof}
\end{theorem}

Now we give informations about the geometry of the leaves of the
distributions $(\ker F_{\ast })$ and $(\ker F_{\ast })^{\perp }$. By using
the relations (\ref{HH1}), (\ref{HH2}) and Lemma \ref{HOR5} and following
same way for the proof of the slant submersions from almost Hermitian
manifolds (see \cite{SAHIN1}), we have

\begin{theorem}
\label{HOR7}Let $F$ be a slant Riemannian submersion from a cosymplectic
manifold $M(\phi ,\xi ,\eta ,g_{M})$ onto a Riemannian manifold $(N,g_{N}).$%
Then the distribution $(\ker F_{\ast })^{\bot }$ defines a totally geodesic
foliation on $M$ if and only if 
\begin{equation*}
g_{M}(\mathcal{H}\nabla _{X}Y,\omega \psi U)=g_{M}(\mathcal{A}_{X}\mathcal{B}%
Y,\omega U)+g_{M}(\mathcal{H}\nabla _{X}\mathcal{C}Y,\omega U)
\end{equation*}%
for any $X,Y\in \Gamma ((\ker F_{\ast })^{\bot })$ and $U\in \Gamma (\ker
F_{\ast }).$
\end{theorem}

\begin{theorem}
\label{HOR8}Let $F$ be a slant Riemannian submersion from a cosymplectic
manifold $M(\phi ,\xi ,\eta ,g_{M})$ onto a Riemannian manifold $(N,g_{N}).$%
Then the distribution $(\ker F_{\ast })$ defines a totally geodesic
foliation on $M$ if and only if 
\begin{equation*}
g_{M}(\mathcal{H}\nabla _{U}\omega \psi V,X)=g_{M}(\mathcal{T}_{U}\omega V,%
\mathcal{B}X)+g_{M}(\mathcal{H}\nabla _{U}\omega V,\mathcal{C}X)
\end{equation*}%
for any $U,V\in \Gamma (\ker F_{\ast })$ and $X\in \Gamma ((\ker F_{\ast
})^{\bot }).$
\end{theorem}

From Theorem \ref{HOR7} and Theorem \ref{HOR8} we obtain

\begin{corollary}
\label{HOR9}Let $F$ be a slant Riemannian submersion from a cosymplectic
manifold $M(\phi ,\xi ,\eta ,g_{M})$ onto a Riemannian manifold $(N,g_{N}).$%
Then $M$ is a locally product Riemannian manifold if and only if and%
\begin{equation*}
g_{M}(\mathcal{H}\nabla _{X}Y,\omega \psi U)=g_{M}(\mathcal{A}_{X}\mathcal{B}%
Y,\omega U)+g_{M}(\mathcal{H}\nabla _{X}\mathcal{C}Y,\omega U)
\end{equation*}%
and%
\begin{equation*}
g_{M}(\mathcal{H}\nabla _{U}\omega \psi V,X)=g_{M}(\mathcal{T}_{U}\omega V,%
\mathcal{B}X)+g_{M}(\mathcal{H}\nabla _{U}\omega V,\mathcal{C}X)
\end{equation*}%
for any $U,V\in \Gamma (\ker F_{\ast })$ and $X\in \Gamma ((\ker F_{\ast
})^{\bot }).$
\end{corollary}

Now we give sufficient and necessary totally geodesic condition for a proper
slant submersion $F$ from cosymplectic manifolds with $\xi \in (\Gamma (\ker
F_{\ast })^{\perp })$.

\begin{theorem}
\label{HOR10}Let $F$ be a slant Riemannian submersion from a cosymplectic
manifold $M(\phi ,\xi ,\eta ,g_{M})$ onto a Riemannian manifold $(N,g_{N})$
with slant angle $\theta $ and $\xi \in (\Gamma (\ker F_{\ast })^{\perp }).$%
Then $F$ is a totally geodesic map if and only if%
\begin{equation*}
g_{M}(\mathcal{H}\nabla _{U}\omega \psi V,X)=g_{M}(\mathcal{T}_{U}\omega V,%
\mathcal{B}X)+g_{M}(\mathcal{H}\nabla _{U}\omega V,\mathcal{C}X)
\end{equation*}%
and%
\begin{equation*}
g_{M}(\mathcal{H}\nabla _{Y}\omega \psi U,X)=g_{M}(\mathcal{A}_{Y}\omega U,%
\mathcal{B}X)+g_{M}(\mathcal{H}\nabla _{Y}\omega U,\mathcal{C}X)
\end{equation*}%
for $U,V\in \Gamma (\ker F_{\ast })$ and $X,Y\in \Gamma ((\ker F_{\ast
})^{\bot }).$
\end{theorem}

\begin{proof}
Using (\ref{5a}) we have 
\begin{equation*}
(\nabla F_{\ast })(X,Y)=0,
\end{equation*}%
for $X,Y\in \Gamma ((\ker F_{\ast })^{\bot }).$Thus it is enough to prove
that $(\nabla F_{\ast })(U,V)=0$ and $(\nabla F_{\ast })(X,U)=0$ for $U,V\in
\Gamma (\ker F_{\ast })$ and $X\in \Gamma ((\ker F_{\ast })^{\bot }).$From (%
\ref{TAN}) and (\ref{HH1}) we obtain%
\begin{equation*}
g_{N}((\nabla F_{\ast })(U,V),F_{\ast }X)=g_{M}(\nabla _{U}\phi \psi
V,X)-g_{M}(\nabla _{U}\omega V,\phi X).
\end{equation*}%
Using again (\ref{TAN}) and (\ref{NOR}) we get%
\begin{eqnarray*}
g_{N}((\nabla F_{\ast })(U,V),F_{\ast }X) &=&g_{M}(\nabla _{U}\psi
^{2}V,X)+g_{M}(\nabla _{U}\omega \psi V,X) \\
&&-g_{M}(\nabla _{U}\omega V,\mathcal{B}X)-g_{M}(\nabla _{U}\omega V,%
\mathcal{C}X).
\end{eqnarray*}%
By Theorem \ref{HOR1}, (\ref{1}), (\ref{2}) and (\ref{HH1}) we get%
\begin{eqnarray*}
g_{N}((\nabla F_{\ast })(U,V),F_{\ast }X) &=&-\cos ^{2}\theta g_{M}(\nabla
_{U}V,X)+g_{M}(\nabla _{U}\omega \psi V,X) \\
&&-g_{M}(\mathcal{T}_{U}\omega V,\mathcal{B}X)-g_{M}(\mathcal{H}\nabla
_{U}\omega V,\mathcal{C}X).
\end{eqnarray*}%
By a simple calculation we have%
\begin{eqnarray}
\sin ^{2}\theta g_{N}((\nabla F_{\ast })(U,V),F_{\ast }X) &=&g_{M}(\nabla
_{U}\omega \psi V,X)-g_{M}(\mathcal{T}_{U}\omega V,\mathcal{B}X)  \label{H9}
\\
&&-g_{M}(\mathcal{H}\nabla _{U}\omega V,\mathcal{C}X).  \notag
\end{eqnarray}%
In a similar way, we get%
\begin{eqnarray}
\sin ^{2}\theta g_{N}((\nabla F_{\ast })(Y,U),F_{\ast }X) &=&g_{M}(\nabla
_{Y}\omega \psi U,X)-g_{M}(\mathcal{A}_{Y}\omega U,\mathcal{B}X)  \label{H10}
\\
&&-g_{M}(\mathcal{H}\nabla _{Y}\omega U,\mathcal{C}X).  \notag
\end{eqnarray}%
Combining (\ref{H9}) and (\ref{H10}) we get requested equations.
\end{proof}

Now we give a sharp inequality between squared mean curvature $\Vert H\Vert
^{2}$ and the scalar curvature $\hat{\tau}$ of fibre through $p\in M^{5}(c)$
such that characteristic vector field $\xi $\textit{\ is horizontal}.

\begin{theorem}
\label{HOR11}Let $F$ be a proper slant Riemannian submersion from a
cosymplectic space form $M^{5}(c)$ onto a Riemannian manifold $(N^{3},g_{N})$%
. Then, we have%
\begin{equation}
\Vert H\Vert ^{2}\geq \frac{1}{4}(\hat{\tau}-\frac{c}{4}(1+3\cos ^{2}\theta
))  \label{horinequ}
\end{equation}%
Moreover, the equality sign of (\ref{horinequ}) holds at a point $p$ of a
fiber if and only if with respect to some suitable slant orthonormal basis $%
\left\{ e_{1},e_{2}=\sec \theta \psi e_{1},e_{3}=\csc \theta
we_{1},e_{4}=\csc \theta we_{2},e_{5}=\xi \right\} $ at $p$, 
\begin{equation*}
T_{11}^{4}=-T_{22}^{4},\text{ and }%
T_{11}^{3}=T_{12}^{3}=T_{22}^{3}=0=T_{ij}^{5}.
\end{equation*}%
where $T_{ij}^{\alpha }=g(\mathcal{T}(e_{i},e_{j}),e_{\alpha })$ for $1\leq
i,j\leq 2$ and $3\leq \alpha \leq 5.$ Here $H$ is mean curvature of fiber.
\end{theorem}

\begin{proof}
By Corollary \ref{HOR3C} and Lemma \ref{HOR4}, we construct a slant
orthonormal basis $\left\{ e_{1},e_{2},e_{3},e_{4},e_{5}\right\} $ defined
by 
\begin{equation}
e_{1},e_{2}=\sec \theta \psi e_{1},e_{3}=\csc \theta we_{1},e_{4}=\csc
\theta we_{2},e_{5}=\xi  \label{pos1}
\end{equation}%
Let $\hat{\tau}$ be scalar curvature of a fibre $F^{-1}(q)$. We choose
arbitrary point $p$ of the fibre $F^{-1}(q).$ Since $\dim (KerF_{\ast })=2$,
we obtain%
\begin{equation}
\hat{\tau}(p)=\hat{K}(e_{1}\wedge e_{2}).  \label{pos2}
\end{equation}%
By (\ref{4D}), (\ref{4E}), (\ref{HH1} $(i)$) and (\ref{CURVATURE}) we get 
\begin{equation}
\hat{K}(e_{1}\wedge e_{2})=\frac{c}{4}(1+3\cos ^{2}\theta
)+T_{11}^{3}T_{22}^{3}+T_{11}^{4}T_{22}^{4}-(T_{12}^{3})^{2}-(T_{12}^{4})^{2},
\label{POS3}
\end{equation}%
where $T_{ij}^{\alpha }=g(\mathcal{T}(e_{i},e_{j}),e_{\alpha })$ for $1\leq
i,j\leq 2$ and $\alpha =3,4,5.$ Using Theorem \ref{HOR1} \ and the relation (%
\ref{H4}) one has 
\begin{equation}
\psi e_{2}=-\cos \theta e_{1}\text{ \ \ and \ \ }\omega e_{2}=\sin \theta
e_{5}.  \label{POS4}
\end{equation}%
From (\ref{F}) we have%
\begin{equation*}
g((\hat{\nabla}_{e_{2}}\psi )e_{2},e_{1})=g(\mathcal{BT}%
_{e_{2}}e_{2},e_{1})-g(\mathcal{T}_{e_{2}}\omega e_{2},e_{1}).
\end{equation*}%
Using (\ref{POS4}) in the last relation we obtain 
\begin{equation}
\sin \theta \lbrack g(\mathcal{T}_{e_{2}}e_{2},e_{3})-g(\mathcal{T}%
_{e_{2}}e_{1},e_{4})]=0.  \label{POS5}
\end{equation}%
Since our submersion is proper, the equation (\ref{POS5}) implies%
\begin{equation*}
T_{22}^{3}=T_{12}^{4}.
\end{equation*}%
Because of $\mathcal{T}_{e_{\alpha }}\xi =0,$ we can choose the unit normal
vector $e_{4}\in \Gamma (\ker (F_{\ast }))^{\perp }$ parallel to the mean
curvature vector $H(p)$ of fibre. Then we have 
\begin{equation*}
\Vert H(p)\Vert ^{2}=\frac{1}{4}(T_{11}^{4}+T_{22}^{4})^{2},\text{ \ \ }%
T_{11}^{3}+T_{22}^{3}=0.
\end{equation*}%
So the relation (\ref{POS3}) becomes 
\begin{equation}
\hat{K}(e_{1}\wedge e_{2})=\frac{c}{4}(1+3\cos ^{2}\theta
)+T_{11}^{4}T_{22}^{4}-(T_{11}^{3})^{2}-(T_{12}^{3})^{2}-(T_{22}^{3})^{2}.
\label{POS6}
\end{equation}%
From the trivial inequality $(\mu +\lambda )^{2}\geq 0$ , one has $(\mu
+\lambda )^{2}\geq \lambda \mu $. Putting $\mu =T_{11}^{4}$ and $\lambda
=T_{22}^{4}$ in the last inequality we find 
\begin{equation}
\Vert H\Vert ^{2}\geq \frac{1}{4}[\hat{K}(e_{1}\wedge e_{2})-\frac{c}{4}%
(1+3\cos ^{2}\theta )].  \label{POS7}
\end{equation}%
By (\ref{pos2}) and the last relation \ we get required inequality.
Moreover, the equality sign of (\ref{horinequ}) holds at a point $p$ of a
fiber if and only if $T_{11}^{4}=-T_{22}^{4},$ \ \ $%
T_{11}^{3}=T_{12}^{3}=T_{22}^{3}=0=T_{ij}^{5}.$
\end{proof}

Recently H.Tastan \cite{TAS} proved that the horizontal distribution of a
Lagrangian submersion from a Kaehlerian manifold to a Riemannian manifold is
integrable and totally geodesic. He also showed that a such submersion is
totally geodesic if and only if it has totally geodesic fibres.

Anti-invariant submersions are special slant submersions with slant angle $%
\theta =\frac{\pi }{4}.$ Now we focus on anti-invariant submersions from a
cosymplectic manifold to a Riemannian manifold such that $(\ker F_{\ast
})^{\bot }$ =$\phi ($ $\ker (F_{\ast }))\oplus \{\xi \}$. In this case we
note that $\ker (F_{\ast })=\phi ((\ker F_{\ast })^{\bot }).$ The authors
investigated such a submersions in \cite{KM}. We will give some additional
results.

By help of (\ref{HHH}) and (\ref{1})-(\ref{4}) we give

\begin{lemma}
\label{HOR12}Let $F$ be a Riemannian submersion from a cosymplectic manifold 
$M(\phi ,\xi ,\eta ,g_{M})$ onto a Riemannian manifold $(N,g_{N})$ such that 
$(\ker F_{\ast })^{\bot }$ = $\phi (\ker (F_{\ast }))\oplus \{\xi \}$. Then%
\begin{equation}
i)\mathcal{T}_{U}\phi E=\phi \mathcal{T}_{U}E\text{ \ \ \ \ \ and \ \ \ \ }%
ii)\mathcal{A}_{X}\phi E=\phi \mathcal{A}_{X}E,  \label{H11}
\end{equation}%
for any $U\in \Gamma (\ker F_{\ast }),X\in \Gamma ((\ker F_{\ast })^{\bot })$
and $E$ $\in \Gamma (TM).$
\end{lemma}

By (\ref{TUW2}) and (\ref{H11}) we obtain following Corollary.

\begin{corollary}
\label{HOR13}Let $F$ be a Riemannian submersion from a cosymplectic manifold 
$M(\phi ,\xi ,\eta ,g_{M})$ onto a Riemannian manifold $(N,g_{N})$ such that 
$(\ker F_{\ast })^{\bot }$ = $\phi (\ker (F_{\ast }))\oplus \{\xi \}$. Then,
for any $X,Y\in \Gamma ((\ker F_{\ast })^{\bot }),$we have 
\begin{equation}
\mathcal{A}_{X}\phi Y=-\mathcal{A}_{Y}\phi X.  \label{H11A}
\end{equation}
\end{corollary}

\begin{theorem}
\label{HOR14}Let $F$ be a Riemannian submersion from a cosymplectic manifold 
$M(\phi ,\xi ,\eta ,g_{M})$ onto a Riemannian manifold $(N,g_{N})$ such that 
$(\ker F_{\ast })^{\bot }$ = $\phi (\ker (F_{\ast }))\oplus \{\xi \}$. Then
the horizontal distribution $(\ker F_{\ast })^{\bot }$ is integrable and
totally geodesic.
\end{theorem}

\begin{proof}
Since tensor field $\mathcal{A=A}_{H}$, it is sufficient to show that $%
\mathcal{A}_{X}=0$ for any $X\in \Gamma ((\ker F_{\ast })^{\bot })$. \ If $\
Y$ and $Z$ are horizontal vector fields on $M,$ we have 
\begin{equation*}
g_{M}(\mathcal{A}_{X}\phi Y,Z)\overset{(\ref{H11A})}{=}-g_{M}(\mathcal{A}%
_{Y}\phi X,Z)\overset{(\ref{H11})}{=}-g_{M}(\phi \mathcal{A}_{Y}X,Z)
\end{equation*}%
\begin{equation*}
\overset{\phi \text{ }anti-sym}{=}g_{M}(\mathcal{A}_{Y}X,\phi Z)\overset{(%
\ref{TUW2})}{=}-g_{M}(\mathcal{A}_{X}Y,\phi Z)\overset{(\ref{4c})}{=}g_{M}(%
\mathcal{A}_{X}\phi Z,Y)
\end{equation*}%
\begin{equation*}
\overset{(\ref{H11A})}{=}-g_{M}(\mathcal{A}_{Z}\phi X,Y)\overset{(\ref{4c})}{%
=}g_{M}(\mathcal{A}_{Z}Y,\phi X)\overset{(\ref{TUW2})}{=}-g_{M}(\mathcal{A}%
_{Y}Z,\phi X)
\end{equation*}%
\begin{equation*}
\overset{(\ref{4c})}{=}g_{M}(\mathcal{A}_{Y}\phi X,Z)\overset{(\ref{H11A})}{=%
}-g_{M}(\mathcal{A}_{X}\phi Y,Z)
\end{equation*}%
So we get 
\begin{equation}
\mathcal{A}_{X}\phi Y=0  \label{H13}
\end{equation}%
which implies $\phi \mathcal{A}_{X}Y=0.$ By (\ref{fi}) we obtain 
\begin{equation}
\mathcal{A}_{X}Y=-g_{M}(\mathcal{A}_{X}Y,\xi )\xi =g_{M}(\mathcal{A}_{X}\xi
,Y)\xi =0.  \label{H14}
\end{equation}%
Since $U$ is a vertical vector field on $M$, $\phi U$ will be horizontal
vector field on $M$. Therefore we obtain $\mathcal{A}_{X}\phi U=0$. Using (%
\ref{fi}) and (\ref{H11}) we have \ 
\begin{equation}
\mathcal{A}_{X}U=0  \label{H15}
\end{equation}%
By virtue of (\ref{H14}) and (\ref{H15}) we get $\mathcal{A}_{X}=0$. The
fact that $(\ker F_{\ast })^{\bot }$ is totally geodesic is obvious from (%
\ref{H13}) and (see: \cite{KM}, Corollary 5).
\end{proof}

Using Theorem (\ref{H14}) and (see: \cite{KM}, Corollary 6 and Theorem 15)
we obtain following Theorem.

\begin{theorem}
\label{HOR15}Let $F$ be a Riemannian submersion from a cosymplectic manifold 
$M(\phi ,\xi ,\eta ,g_{M})$ onto a Riemannian manifold $(N,g_{N})$ such that 
$(\ker F_{\ast })^{\bot }$ = $\phi (\ker (F_{\ast }))\oplus \{\xi \}$. Then $%
F$ is a totally geodesic map if and only if it has totally geodesic fibers.
\end{theorem}

Since horizontal distribution for a slant Riemannian submersion from a
cosymplectic manifold $M(\phi ,\xi ,\eta ,g_{M})$ onto a Riemannian manifold 
$(N,g_{N})$ with $(\ker F_{\ast })^{\bot }$ = $\phi (\ker (F_{\ast }))\oplus
\{\xi \}$ is integrable, the equation (\ref{4F}) reduces to 
\begin{equation}
R(Y,W,V,X)=g_{M}((\nabla _{X}\mathcal{T})(V,W),Y)-g_{M}(\mathcal{T}_{V}X,%
\mathcal{T}_{W}Y)  \label{H16}
\end{equation}%
for any $X,Y,Z\in \Gamma ((\ker F_{\ast })^{\bot })$, $V,W\in \Gamma (\ker
F_{\ast })$.

From (\ref{HOLOMORPHIC}) and (\ref{H16}) we give

\begin{theorem}
\label{HOR16}Let $F$ be a Riemannian submersion from a cosymplectic manifold 
$M(\phi ,\xi ,\eta ,g_{M})$ onto a Riemannian manifold $(N,g_{N})$ such that 
$(\ker F_{\ast })^{\bot }$ = $\phi (\ker (F_{\ast }))\oplus \{\xi \}$. Then
the $\phi $-sectional curvature $H$ of $M$ satisfies 
\begin{equation}
H(V)=g_{M}((\nabla _{\phi V}\mathcal{T})(V,V),\phi V)-g_{M}(\mathcal{T}_{V}V,%
\mathcal{T}_{V}V),  \label{H17}
\end{equation}%
\begin{equation}
H(X)=g_{M}((\nabla _{X}\mathcal{T})(\phi X,\phi X),X)-g_{M}(\mathcal{T}%
_{\phi X}X,\mathcal{T}_{\phi X}X),  \label{H18}
\end{equation}%
for any $X,Y,Z\in \Gamma ((\ker F_{\ast })^{\bot })$, $V,W\in \Gamma (\ker
F_{\ast })$.
\end{theorem}

It is well known from \cite{ESC} that if the tensor field $\mathcal{T}$ is
parallel, i.e.$\nabla \mathcal{T}=0$ for a Riemannian submersion. then $%
\mathcal{T}=0$. From Theorem \ref{HOR16} we have

\begin{corollary}
\label{HOR17}Let $F$ be a Riemannian submersion from a cosymplectic manifold 
$M(\phi ,\xi ,\eta ,g_{M})$ onto a Riemannian manifold $(N,g_{N})$ such that 
$(\ker F_{\ast })^{\bot }$ = $\phi (\ker (F_{\ast }))\oplus \{\xi \}$. If
the tensor the tensor field $\mathcal{T}$ is parallel, then the $\phi $%
-sectional curvature $H$ of $M$ vanishes.
\end{corollary}

\begin{corollary}
\label{HOR18}Let $M(c\neq 0)$ be a cosymplectic space form . Then there is
no Riemannian submersion $F$ with totally geodesic fibres from a
cosymplectic space form $M(c\neq 0)$ onto a Riemannian manifold $(N,g_{N})$
such that $(\ker F_{\ast })^{\bot }$ = $\phi (\ker (F_{\ast }))\oplus \{\xi
\}$.
\end{corollary}

\end{document}